\documentclass[12pt,leqno,draft]{article}
\usepackage{amsfonts}
\usepackage{amsmath, amsthm, amsfonts, amssymb, color}
\usepackage{mathrsfs}
\usepackage{color}
\newtheorem{thm}{Theorem}[section]

\newtheorem{lem}[thm]{Lemma}

\theoremstyle{definition}

\newcommand{\scr}[1]{\mathscr #1}
\definecolor{wco}{rgb}{0.5,0.2,0.3}

\numberwithin{equation}{section} \theoremstyle{remark}

\newcommand{\ua}{\uparrow}

\title{{\bf  Path-Distribution Dependent SDEs with Singular Coefficients}\footnote{Supported in
 part by  NNSFC (11801406).} }
\author{
{\bf   Xing Huang $^{a)}$}\\
\footnotesize{ a)Center for Applied Mathematics, Tianjin
University, Tianjin 300072, China}\\
\footnotesize{  xinghuang@tju.edu.cn}}
\begin{document}
\allowdisplaybreaks
\def\R{\mathbb R}  \def\ff{\frac} \def\ss{\sqrt} \def\B{\mathbf
B} \def\W{\mathbb W}
\def\N{\mathbb N} \def\kk{\kappa} \def\m{{\bf m}}
\def\ee{\varepsilon}\def\ddd{D^*}
\def\dd{\delta} \def\DD{\Delta} \def\vv{\varepsilon} \def\rr{\rho}
\def\<{\langle} \def\>{\rangle} \def\GG{\Gamma} \def\gg{\gamma}
  \def\nn{\nabla} \def\pp{\partial} \def\E{\mathbb E}
\def\d{\text{\rm{d}}} \def\bb{\beta} \def\aa{\alpha} \def\D{\scr D}
  \def\si{\sigma} \def\ess{\text{\rm{ess}}}
\def\beg{\begin} \def\beq{\begin{equation}}  \def\F{\scr F}
\def\Ric{\text{\rm{Ric}}} \def\Hess{\text{\rm{Hess}}}
\def\e{\text{\rm{e}}} \def\ua{\underline a} \def\OO{\Omega}  \def\oo{\omega}
 \def\tt{\tilde} \def\Ric{\text{\rm{Ric}}}
\def\cut{\text{\rm{cut}}} \def\P{\mathbb P} \def\ifn{I_n(f^{\bigotimes n})}
\def\C{\scr C}      \def\aaa{\mathbf{r}}     \def\r{r}
\def\gap{\text{\rm{gap}}} \def\prr{\pi_{{\bf m},\varrho}}  \def\r{\mathbf r}
\def\Z{\mathbb Z} \def\vrr{\varrho} \def\ll{\lambda}
\def\L{\scr L}\def\Tt{\tt} \def\TT{\tt}\def\II{\mathbb I}
\def\i{{\rm in}}\def\Sect{{\rm Sect}}  \def\H{\mathbb H}
\def\M{\scr M}\def\Q{\mathbb Q} \def\texto{\text{o}} \def\LL{\Lambda}
\def\Rank{{\rm Rank}} \def\B{\scr B} \def\i{{\rm i}} \def\HR{\hat{\R}^d}
\def\to{\rightarrow}\def\l{\ell}\def\iint{\int}
\def\EE{\scr E}\def\Cut{{\rm Cut}}
\def\A{\scr A} \def\Lip{{\rm Lip}}
\def\BB{\scr B}\def\Ent{{\rm Ent}}\def\L{\scr L}
\def\R{\mathbb R}  \def\ff{\frac} \def\ss{\sqrt} \def\B{\mathbf
B}
\def\N{\mathbb N} \def\kk{\kappa} \def\m{{\bf m}}
\def\dd{\delta} \def\DD{\Delta} \def\vv{\varepsilon} \def\rr{\rho}
\def\<{\langle} \def\>{\rangle} \def\GG{\Gamma} \def\gg{\gamma}
  \def\nn{\nabla} \def\pp{\partial} \def\E{\mathbb E}
\def\d{\text{\rm{d}}} \def\bb{\beta} \def\aa{\alpha} \def\D{\scr D}
  \def\si{\sigma} \def\ess{\text{\rm{ess}}}
\def\beg{\begin} \def\beq{\begin{equation}}  \def\F{\scr F}
\def\Ric{\text{\rm{Ric}}} \def\Hess{\text{\rm{Hess}}}
\def\e{\text{\rm{e}}} \def\ua{\underline a} \def\OO{\Omega}  \def\oo{\omega}
 \def\tt{\tilde} \def\Ric{\text{\rm{Ric}}}
\def\cut{\text{\rm{cut}}} \def\P{\mathbb P} \def\ifn{I_n(f^{\bigotimes n})}
\def\C{\scr C}      \def\aaa{\mathbf{r}}     \def\r{r}
\def\gap{\text{\rm{gap}}} \def\prr{\pi_{{\bf m},\varrho}}  \def\r{\mathbf r}
\def\Z{\mathbb Z} \def\vrr{\varrho} \def\ll{\lambda}
\def\L{\scr L}\def\Tt{\tt} \def\TT{\tt}\def\II{\mathbb I}
\def\i{{\rm in}}\def\Sect{{\rm Sect}}  \def\H{\mathbb H}
\def\M{\scr M}\def\Q{\mathbb Q} \def\texto{\text{o}} \def\LL{\Lambda}
\def\Rank{{\rm Rank}} \def\B{\scr B} \def\i{{\rm i}} \def\HR{\hat{\R}^d}
\def\to{\rightarrow}\def\l{\ell}
\def\8{\infty}\def\I{1}\def\U{\scr U}
\maketitle

\begin{abstract} In this paper, existence and uniqueness are proved for path-dependent McKean-Vlasov type SDEs with integrability conditions. Gradient estimates and Harnack type inequalities are derived in the case that the coefficients are Dini continuous in the space variable. These generalize the corresponding results derived for classical functional SDEs with singular coefficients.

\end{abstract} \noindent
 \vskip 2cm

\section{Introduction}
The distribution dependent SDEs  can be used to  characterize  nonlinear Fokker-Planck equations, see \cite{CD,CR,MV,SZ} and references within for McKean-Vlasov type SDEs, and \cite{ CA,DV1,DV2} and references within for Landau type equations, see also for the path-distribution dependent SDEs with regular conditions.

Recently, \cite{HW18} studied the existence and uniqueness of distribution dependent SDEs with singular coefficients. The Harnack, shift Harnack inequalities and gradient estimate are also investigated in \cite{HW18}. \cite{RZ} also obtains the existence and uniqueness, estimate of heat kernel for singular distribution dependent SDEs.
 For more results on distribution independent SDEs with singular coefficients, one can see \cite{GM,KR,Z2,W16} and references therein, where the Zvonkin transform in \cite{AZ} plays an important role.

The purpose of this paper is to extend results in \cite{HW18} to path-distribution dependent SDEs with singular drift. Firstly, due to the distribution dependence, the Girsanov transform, which is a useful tool to prove the existence of weak solution for the classical SDEs is unavailable. Thus, compared to the classical SDEs with singular drift, we will pay more attention in the proof of existence of weak solution. In other words, we will apply an approximation technique similar to that in \cite{HW18,RZ} to obtain weak existence. However, the path-distribution dependent drift will add new difficulty, see the proof of Theorem \ref{T1.1} (1) below. With the weak existence in hand, if using a fixed distribution $\mu_t$ to replace the law of solution $\L_{X_t}$, the SDE \eqref{E10} has strong uniqueness,  then a strong solution for SDE \eqref{E10} can be obtained. To prove the strong uniqueness, we will again use the technique in \cite{HW18}, i.e. we first identify the distributions of given two solutions, so that these  solutions  solve the common reduced classical SDE, and thus, the pathwise uniqueness follows from existing argument developed for the classical SDEs. The essential difficulty lies in identifying the distributions of two solutions of \eqref{E10}. Finally, gradient estimates and Harnack type inequalities can be proved by Zvonkin's transform as in the regular situation considered in \cite{FYW1}.

Fix a constant $r> 0$, let $\C= C([-r,0];\mathbb{R}^d)$ be equipped with the uniform norm $\|\xi\|_{\C} =:\sup_{s\in[-r,0]} |\xi(s)|$. For any $f\in C([-r,\infty);\mathbb{R}^d)$, $t\geq 0$, define $f_t \in \C$ as $f_t(s)=f(t+s), s\in [-r,0]$, which is called the segment process.

Let $\scr P$ be the set of all probability measures on $\C$. Consider the following path-distribution dependent SDE on $\R^d$:
\beq\label{E10} \d X(t)= B(t,X_t, \L_{X_t})\d t+ b(t,X(t),\L_{X_t})+\si(t,X(t), \L_{X_t})\d W(t),\end{equation}
where $W(t)$ is the $d$-dimensional Brownian motion on a complete filtration probability space $(\OO,\{\F_t\}_{t\ge 0},\P)$, $\L_{X_t}$ is the law of $X_t$, and
$$b: \R_+\times\R^d\times \scr P\to \R^d,\ \ B: \R_+\times\C\times \scr P\to \R^d,\ \ \si: \R_+\times\R^d\times \scr P\to \R^d\otimes\R^d$$ are measurable. We use $\L_\xi|\tt \P$ to denote the law of a random variable $\xi$ under the probability $\tt\P$.

Throughout the paper, we use $\|\cdot\|_\infty$ to denote the uniform norm.
The remainder of the paper is organized as follows. In Section 2 we summarize the main results of the paper. To prove these results,   some preparations are addressed in Section~3, including a new Krylov's estimate, one lemma on convergence of stochastic processes, and a result on the existence of strong solutions for distribution dependent SDEs. Finally,  the main results are proved in Sections 4 and 5.


\section{Main Results}
Let $\theta\in [1,\infty)$, we will consider the SDE \eqref{E10} with initial distributions in the class
$$\scr P_\theta := \big\{\mu\in \scr P: \mu(\|\cdot\|_{\C}^\theta)<\infty\big\}.$$    It is well known that
$\scr P_\theta$ is a Polish space under the Wasserstein distance
$$\W_\theta(\mu,\nu):= \inf_{\pi\in \mathbf{C}(\mu,\nu)} \bigg(\int_{\C\times\C} \|\xi-\eta\|_{\C}^\theta \pi(\d \xi,\d \eta)\bigg)^{\ff 1 {\theta}},\ \ \mu,\nu\in \scr P_{\theta},$$ where $\mathbf{C}(\mu,\nu)$ is the set of all couplings of $\mu$ and $\nu$.  Moreover,   the topology induced by $\W_\theta$ on $\scr P_\theta$ coincides with the weak topology.

In the following three subsections,  we state our main results   on the existence, uniqueness and Harnack type inequalities respectively  for the distribution dependent SDE \eqref{E10}.

\subsection{Existence and Uniqueness}
We will fix a constant $T>0$, and only consider solutions of \eqref{E10} up to time $T$.
 For a measurable function $f$ defined on $[0,T]\times\mathbb{R}^d$, let
$$\|f\|_{L^q_p(s,t)}=\left(\int_s^t\left(\int_{\mathbb{R}^d}|f(r,x)|^p\d x\right)^{\frac{q}{p}}\d r\right)^{\frac{1}{q}}, \ \ p,q\ge 1, 0\le s\le t\le T. $$ When $s=0$, we simply denote   $\|f\|_{L^q_p(0,t)}=\|f\|_{L^q_p(t)}$. A key step in the study of singular SDEs is to establish Krylov type estimate (see for instance \cite{KR}).
For later use we introduce the following class of number pairs $(p,q)$:
   $$\scr K:=\Big\{(p,q)\in (1,\infty)\times(1,\infty):\   \ff d p +\ff 2 q<2\Big\}.$$
To construct a weak solution of \eqref{E10} by using approximation argument as in \cite{GM, HW18, MV, RZ}, we need the following conditions.

\beg{enumerate} \item[$(H^\theta)$] The following assumptions hold for some $\theta\ge 1$.
\item[$\ (1)$]  For $ \mu\in \scr P_\theta$ and $\mu^n\to \mu$ in $\scr P_\theta$,
$$  \lim_{n\to\infty}  \big\{ |b(t,x,\mu^n)-b(t,x,\mu)|+ \|\sigma(t,x,\mu^n)- \sigma(t,x,\mu)\| \big\} =0,\ \ \text{a.e.}\ \  (t,x)\in [0,T]\times\R^d.  $$
\item[$\ (2)$]  There exist $K>1$,   $(p,q)\in \scr K$ and  nonnegative $F\in  L_p^q(T)$
such that
 $$ |b(t,x,\mu)|^2\le F(t,x)+K,\ \ K^{-1} I \le (\sigma\sigma^\ast)(t,x,\mu)\le K I$$
 for all $(t,x,\mu)\in [0,T]\times \R^d\times \scr P_\theta.$
 \item[$\ (3)$] $B$ is bounded and for any $t\in[0,T], \xi\in\C$, $B(t,\xi,\cdot)$ is continuous on $\scr P_\theta$. Moreover, there exists a constant $L_0>0$ such that
 \begin{equation}\label{Lipx}
 |B(t,\xi,\mu)-B(t,\bar{\xi},\mu)|\leq L_0\|\xi-\bar{\xi}\|_{\C}, \ \ t\in[0,T], \xi,\bar{\xi}\in\C,\mu\in\scr P_\theta.
 \end{equation}
\end{enumerate}
 Recall that a continuous function $f$ on $\mathbb{R}^d$ is called weakly differentiable, if there exists (hence unique) $h\in L^1_{loc}(\mathbb{R}^d)$ such that
$$\int_{\mathbb{R}^d}(f\Delta g)(x)\d x=-\int_{\mathbb{R}^d}\langle h,\nabla g\rangle (x)\d x, \ \ g\in C_0^\infty(\mathbb{R}^d).$$
In this case, we write $h=\nabla f$ and call it the weak gradient of $f$.

 The main result in this part is the following.

\begin{thm}\label{T1.1} Assume $(H^\theta)$ for some constant $\theta\ge 1$.   Let $X_0$ be an $\F_0$-measurable   random variable on $\C$ with $\mu_0:=\L_{X_0}\in \scr P_\theta$. Then the following assertions hold.
\beg{enumerate} \item[$(1)$] The SDE $\eqref{E10}$ has a weak solution with initial distribution $\mu_0$ satisfying $ \L_{X_\cdot} \in C([0,T];\scr P_\theta)$.
\item[$(2)$] If $\sigma$ is uniformly continuous in $x\in\mathbb{R}^d$ uniformly with respect to $(t,\mu)\in[0,T]\times\scr P_{\theta},$ and
 for any $\mu(\cdot)\in C([0,T]; \scr P_\theta)$, $b^\mu(t,x):= b(t, x, \mu_t)$ and $\si^\mu(t,x):= \si(t,x,\mu_t)$ satisfy
$| b^\mu|^2+\|\nn \si^\mu\|^2 \in L_p^q(T)$ for some $(p,q)\in \scr K$, where $\nn$ is the weak gradient in the space variable $x\in \R^d$, then the SDE \eqref{E10} has a strong solution satisfying $\L_{X_\cdot}\in C([0,T];\scr P_\theta)$.
\item[$(3)$] If, in addition to the condition in $(2)$,  there exists a  constant  $L\,>0$ such that
\beq\label{LIP} \|\si(t,x,\mu)-\si(t,x,\nu)\|+ |b(t,x,\mu)-b(t,x,\nu)|\le L\, \W_\theta(\mu,\nu)\end{equation}
and
\beq\label{LIP'} |B(t,\xi,\mu)-B(t,\xi,\nu)|\le L\W_\theta(\mu,\nu)\end{equation}
holds for all $ \mu,\nu\in \scr P_\theta$ and $ (t,x)\in [0,T]\times \R^d, \xi\in\C,$  then  the strong solution is unique.
\end{enumerate}
\end{thm}

When $B$, $b$ and $\si$ do not depend on the distribution, Theorem \ref{T1.1}  reduces back to the corresponding results derived for classical functional SDEs with singular coefficients, see for instance \cite{B}  and references within.


\subsection{Harnack Inequality}

In this subsection, we  investigate the dimension-free Harnack inequality introduced in \cite{RW10} for \eqref{E10}, see \cite{Wbook} and references within for general results on these type Harnack inequalities and applications.  We establish Harnack inequalities for $P_tf$ using  coupling by change of measures (see for instance \cite[\S 1.1]{Wbook}). To this end, we need to assume that the noise part is distribution-free; that is, we consider the following special version of \eqref{E10}:
\beq\label{E11} \d X(t)= B(t,X_t,\L_{X_t})\d t+b(t,X(t),\L_{X_t})\d t +\si(t,X(t))\d W(t),\ \ t\in [0,T].\end{equation}

As in \cite{HRW}, we define $P_tf(\mu_0)$ and $P_t^*\mu_0$ as follows:
$$(P_tf)(\mu_0)= \int_{\C} f  \d(P_t^*\mu_0)= \E f(X_t(\mu_0)),\ \ f\in \B_b(\C), t\in [0,T],  \mu_0\in \scr P_2,$$ where $X_t(\mu_0)$ solves \eqref{E11} with
  $\L_{X_0}=\mu_0.$
Let
$$\D=\bigg\{\phi:  [0,\infty)\to [0,\infty)  \text{\ is\ increasing},  \phi^2 \text{\ is\ concave,} \int_0^1\ff{\phi(s)}s\d s<\infty\bigg\}.$$
\beg{rem}\label{Dini}
 The condition $\int_0^1{\frac{\phi(s)}{s}\d s}<\infty$ is   known as the Dini condition. Obviously, $\D$ contains $\phi(s)=s^{\alpha}$ for any $\alpha\in(0,\frac{1}{2})$. Moreover, it also contains $\phi(s):=\frac{1}{\log^{1+\delta}(c+s^{-1})}$ for constants $\delta >0$ and large enough $c>0$ such that $\phi^{2}$ is concave.
 \end{rem}
We will need the following assumption.
\beg{enumerate}
\item[$\bf{(H)}$] $\|b\|_{\infty}+\|B\|_{\infty}<\infty$ and there exist a constant $K>1$ and $\phi\in \D$ such that for any
$t\in[0,T],\ x,y\in \mathbb{R}^d,$ and $\mu,\nu\in \scr P_2$, $\xi,\bar{\xi}\in\C$,
\beq\label{H1}
K^{-1}I\leq(\sigma\sigma^{\ast})(t,x) \leq KI,\
\|\sigma(t,x)-\sigma(t,y)\|^2_{\mathrm{HS}}\leq K|x-y|^2,
\end{equation}
\beq\label{b-phi}
|b(t,x,\mu)-b(t,y,\nu)|\leq \phi(|x-y|)+K\mathbb{W}_2(\mu,\nu),
\end{equation}
\beq\label{Blip}
|B(t,\xi,\mu)-B(t,\bar{\xi},\nu)|\leq K(\|\xi-\bar{\xi}\|_{\C}+\mathbb{W}_2(\mu,\nu)).
\end{equation}
\end{enumerate}
\beg{thm}\label{T3.1} Assume {\bf (H)}. \beg{enumerate} \item[$(1)$] There exists a  constant $C>0$ such that
  \beq\label{LH2}(P_{t}\log f)(\nu_0)\le  \log (P_{t}f)(\mu_0)+ \ff{C}{t-r}\W_2(\mu_0,\nu_0)^2\end{equation}
  for any $t\in (r,T],\mu_0,\nu_0\in\scr P_2,  f\in \B_b^+(\C)$ with $f\geq 1.$ Consequently, for any different $\mu_0,\nu_0\in \scr P_2$, and any $f\in \B_b(\C)$,
\begin{align} \label{GE}\ff{|(P_{t}f)(\mu_0)-(P_{t}f)(\nu_0)|^2}{\W_2(\mu_0,\nu_0)^2} \le  \ff{2C}{t-r}\sup_{\nu\in \mathbf{B}(\mu_0, \W_2(\mu_0,\nu_0))}\big\{(P_{t} f^2)(\nu)- (P_{t} f)^2(\nu)\big\},
\end{align}
where $\mathbf{B}(\mu_0, \W_2(\mu_0,\nu_0)):=\{\nu\in\scr P_2, \W_2(\mu_0,\nu)<\W_2(\mu_0,\nu_0)\}$.
 \item[$(2)$] There exist constants $p_0>1$ and $c_1,c_2>0$,  such that for any $p>p_0$, $t\in (r,T],  f\in \B_b^+(\C)$ and $ \mu_0,\nu_0\in\scr P_2$,

   \beq\label{H2'}(P_{t}f)^p(\nu_0)\le  (P_{t}f^p)(\mu_0)\Big(\E \Big(\e^{H_2(p,t)\big(1+\ff{|X(0-Y(0)|^2}{t-r} +\|X_0-Y_0\|_{\C}^2\big)}\Big)\Big)^p\end{equation}
   holds for  $\F_0$-measurable random variables $X_0,Y_0$ satisfying $\L_{X_0}=\mu_0$, $\L_{Y_0}=\nu_0$ . \end{enumerate}
\end{thm}

\subsection{Shift Harnack Inequality}

In this section, we establish the shift Harnack inequality for $P_t$  introduced in \cite{W14a}. To this end, we assume that $\sigma(t,x,\mu)$ does not depend on $x$. So SDE \eqref{E10} becomes
\beq\label{E5} \d X(t)= B(t,X_t,\L_{X_t})\d t+b(t,X(t),\L_{X_t})\d t +\si(t,\L_{X_t})\d W(t),\ \ t\in [0,T].\end{equation}


\beg{thm}\label{T5.1} Let $\si: [0,T]\times \scr P_2\to \R^d\otimes \R^d$ be  measurable such that
$\si$ is invertible with $\|\si\|_{\infty}+\|\si^{-1}\|_{\infty}<\infty$, and $b$, $B$ satisfy the corresponding conditions in {\bf (H)}.
\beg{enumerate} \item[$(1)$] For any $p>1, t\in (r,T], \mu_0\in \scr P_2, \eta\in C^1([-r,0],\mathbb{R}^d)$ and $f\in \B_b^+(\C)$,
\beg{align*}(P_{t}f)^p(\mu_0)\le &(P_{t}f^p(\eta+\cdot))(\mu_0)\times \exp\bigg[\ff{p}{2(p-1)}\beta(T,\eta,r)\bigg].\end{align*}
where
$$\beta(T,\eta,r)=\|\sigma^{-1}\|^2_\infty \left(C\ff {|\eta(-r)|^2}{T-r}+C\int_{-r}^0|\eta'(s)|^2\d s+CT\phi^2\left(C\|\eta\|_{\C}\right)+CT\|\eta\|_{\C}^2\right),$$
and $C>0$ is a constant.
Moreover, for any $f\in \B_b^+(\C)$ with $f\geq 1$,
$$(P_{t}\log f)(\mu_0)\le \log (P_{t} f(\eta+\cdot))(\mu_0)+\beta(T,\eta,r).$$
\end{enumerate} \end{thm}

\section{Preparations}
 We first recall Krylov's estimate in the study of SDEs.
\beg{defn}[Krylov's Estimate]  \emph{An $\F_t$-adapted process $\{X(s)\}_{0\le s\le T}$ is said to satisfy $K$-estimate, if   for any $(p,q)\in \scr K$, there exist    constants $\dd\in (0,1)$ and
$C>0$ such that for any nonnegative measurable function $f$ on $[0,T]\times \R^d$,
\beq\label{KR1} \E\bigg(\int_s^t f(r,X(r)) \d r\Big| \F_s\bigg) \le C (t-s)^\dd \|f\|_{L_p^q(T)},\ \ \ 0\le s\le t\le T. \end{equation}}
\end{defn}
We note that \eqref{KR1} implies the following Khasminskii type estimate, see for instance \cite[Lemma 3.5]{XZ} and its proof: there exists a constant $c>0$ such that
\beq\label{APP3} \E\bigg(\bigg(\int_s^t f(r,X(r)) \d r\bigg)^n\Big| \F_s\bigg) \le c n! (t-s)^{\dd n}\|f\|_{L_p^q(T)}^n,\ \ \ 0\le s\le t\le T, \end{equation} and
for any $\ll>0$ there exists a constant $\LL=\LL(\ll,\dd,c)>0$  such that
\beq\label{KR2}   \E\big(\e^{\ll\int_0^T f(r,X(r)) \d r}\big| \F_s\big) \le \e^{\LL  \left(1+\|f\|_{L_p^q(T)}\right)},\ \ s\in[0,T]. \end{equation}

We first present a new result on Krylov's estimate, then recall one lemma from \cite{GM} for the construction of weak solution, and finally introduce one lemma on the relation between existence of strong and weak solutions.

\subsection{Krylov's Estimate}

Consider the following SDE on $\R^d$:
\beq\label{EN} \d X(t)= B(t,X_t)\d t+b(t,X(t))\d t + \si(t,X(t))\d W(t),\ \ t\in [0,T].\end{equation}
\beg{lem}\label{KK} Let $T>0$, and let $p,q\in (1,\infty)$ with $\ff d p +\ff 2 q<1$.  Assume that $\sigma(t,x)$ is uniformly continuous in $x\in\mathbb{R}^d$ uniformly with respect to $t\in[0,T]$, and that for a constant $K>1$ and some nonnegative function $F\in L_p^q(T)$
such that \beq\label{APP1} K^{-1}I\le \si(t,x)\si^*(t,x)\le K I,\ \ (t,x)\in [0,T]\times \R^d,\end{equation}
 \beq\label{APP2} |b(t,x)| \le K+ F(t,x),\ \ (t,x)\in [0,T]\times \R^d.\end{equation}
  \begin{equation*}|B(t,\xi)| \le K,\ \ (t,\xi)\in [0,T]\times \C.\end{equation*}
 Then for any $(\aa,\bb)\in \scr K$, there exist constants  $C=C(\dd,K, \aa,\bb, \|F\|_{L_p^q(T)})>0$  and $\dd=\dd(\aa,\bb)>0$, such that for any $s_0\in [0,T)$ and any solution $(X_{s_0,t})_{t\in [s_0,T]}$ of $\eqref{EN}$ from time $s_0$,
\beq\label{APP'}\E\bigg[\int_s^t |f|(r, X_{s_0,r}) \d r\Big| \F_s\bigg]\le   C  (t-s)^{\dd}\|f\|_{L_{\aa}^{\bb}(T)},\ s_0\leq s<t\leq T, f\in L_{\aa}^{\bb}(T).\end{equation}
 \end{lem}
 \beg{proof} Let $$ \tilde{W}(\cdot)=W(\cdot)+\int_0^\cdot B(r, X_{s_0,r})\d r.$$
 Since $B$ is bounded, by Girsanov's theorem, $\tilde{W}$ is a $d$-dimensional Brownian motion on $[0,T]$ under $\mathbb{Q}=R(T)\mathbb{P}$, where
\begin{align*}
R(T)&=\exp\bigg[-\int_{s_0}^T\<B(r, X_{s_0,r}),\d W(r)\>-\frac{1}{2}\int_{s_0}^T |B(r, X_{s_0,r})|^2\d r\bigg].
\end{align*}
Moreover, the boundedness of $B$ implies $\E R(T)^{-1}<\infty$. Thus, $(\{X_{s_0,}(r)\}_{r\in[s_0,T] }, \tilde{W})$ is a weak solution
\beq\label{EN'} \d X(t)= b(t,X(t))\d t + \si(t,X(t))\d W(t).\end{equation}
By \cite[Lemma 3.1]{HW18}, there exists a constant $C=C(\dd,\bar{K}, \aa,\bb)>0$ and $\dd=\dd(\aa,\bb)>0$ such that
\beq\label{APP3'}
\E^{\mathbb{Q}}\bigg[\int_s^t |f|(r, X_{s_0,}(r)) \d r\Big| \F_s\bigg]\le   C  (t-s)^{\dd}\|f\|_{L_{\aa}^{\bb}(T)},\ s_0\leq s<t\leq T, f\in L_{\aa}^{\bb}(T).\end{equation}
This together with \eqref{APP3} and H\"{o}lder inequality implies that
\begin{align*}
&\left(\E\bigg[\int_s^t |f|(r, X_{s_0,}(r)) \d r\Big| \F_s\bigg]\right)^2=\E R(T)^{-1}\times \E^\Q\bigg[\left(\int_s^t |f|(r, X_{s_0,}(r)) \d r\right)^2\Big| \F_s\bigg]\\
&\leq C\E^\Q\bigg[\left(\int_s^t |f|(r, X_{s_0,}(r)) \d r\right)^2\Big| \F_s\bigg]\\
&\le   C  (t-s)^{2\dd}\|f\|^2_{L_{\aa}^{\bb}(T)},\ s_0\leq s<t\leq T,, f\in L_{\aa}^{\bb}(T).\end{align*}
Then the proof is finished.
 \end{proof}

\subsection{Convergence  of Stochastic Processes}
To prove Theorem \ref{T1.1}(1), we will use the following lemma due to \cite[Lemma 5.1]{GM}.
\begin{lem}\label{PC}  Let $\{\psi^n\}_{n\geq 1}$ be a sequence of $d$-dimensional processes defined on
some probability space. Assume that there exists a constant $\alpha>0$ such that
\begin{align}\label{Ub}\sup_{n\geq 1}\sup_{t\in[0,T]}\E(|\psi^n_t|^\alpha)<\infty,
\end{align}
and for any $\vv>0$,
\begin{align}\label{ETC}
\lim_{\theta\to0}\sup_{n\geq 1}\sup_{s,t\in[0,T], |t-s|\leq \theta}\E\left(|\psi^n_t-\psi^n_s|^\alpha\right)=0.
\end{align}
  Then there exist a sequence $\{n_k\}_{k\geq 1}$, a probability space $(\tilde{\Omega},\tilde{\F}, \tilde{\P})$ and
stochastic processes $\{X_t, X^k_t\}_{t\in [0,T]} (k \geq  1)$, such that for every $t\in [0,T]$, $\L_{\psi^{n_k} _t}|\P=\L_{X^k_t}|\tilde{\P}$, and $\tilde{\P}$-a.s. $X^k$ converges to $X$ as $k\to\infty$. Moreover, $X^k$ converges to $X$ weakly. 
\end{lem}
\begin{proof} \eqref{Ub} and \eqref{ETC} imply that $\{\psi^n\}_{n\geq 1}$ is tight. Then there exists a subsequence $\{m_l\}_{l\geq 1}$ such that $\{\psi^{m_l}\}_{l\geq1}$ is weakly convergent. By
Skorohod representation theorem, there exists a subsequence $\{n_k\}_{k\geq 1}$ of $\{m_l\}_{l\geq 1}$, a probability space $(\tilde{\Omega},\tilde{\F}, \tilde{\P})$ and
stochastic processes $\{X_t, X^k_t\}_{t\in [0,T]} (k \geq  1)$, such that $\L_{\psi^{n_k} }|\P=\L_{X^k}|\tilde{\P}$, and $\tilde{\P}$-a.s. $X^k$ converges to $X$ as $k\to\infty$. It is easy to see that $X^k$ converges to $X$ weakly. The proof is completed. 
\end{proof}

\subsection{Relation between existence of Strong and Weak Solutions}
We present a result on the existence of strong solutions deduced from weak solutions.
Consider the following SDE
\begin{align}\label{class0}
\d X(t)= \hat{B}(t,X_t,\L_{X_t})\,\d t+\hat{\sigma}(t,X_t,\L_{X_t})\,\d W(t),\ \ 0\le t\le T,
\end{align}
where $\hat{B}:[0,T]\times \C\times\scr P\to\mathbb{R}^d$ and $\hat{\sigma}:[0,T]\times \C\times\scr P\to\mathbb{R}^d\otimes\mathbb{R}^d$ are measurable.
\begin{lem}\label{SS} Let $(\bar\Omega, \bar\F_t,\bar W(t),\bar\P)$ and $\bar{X}_t$ be a weak solution to \eqref{class0} with $\mu_t:=\L_{\bar X_t}|\bar\P$. If the SDE
\begin{align}\label{class}
\d X(t)= \hat{B}(t,X_t,\mu_t)\,\d t+\hat{\sigma}(t,X_t,\mu_t)\,\d W(t),\ \ 0\le t\le T
\end{align}
has a unique strong solution $X_t$ up to life time with $\L_{X_0}=\mu_0$, then   \eqref{class0} has a strong solution.
\end{lem}
\begin{proof} Since $\mu_t= \scr L_{\bar X_t}|\bar \P$,   $\bar{X}_t$ is a weak solution to \eqref{class}. By Yamada-Watanabe principle, the strong uniqueness of \eqref{class} implies the weak uniqueness, so that $X_t$ is nonexplosive with    $\L_{X_t}=\mu_t, t\ge 0$. Therefore, $X_t$ is a strong solution to \eqref{class0}.
\end{proof}
\section{Proofs of Theorem \ref{T1.1}}
\subsection{Proof of Theorem \ref{T1.1}(1)-(2)}

  According to \cite[Theorem 1.4]{B},  the condition in Theorem \ref{T1.1}(2) implies that the SDE \eqref{class} has a unique strong solution. So,
  by Lemma \ref{SS}, Theorem \ref{T1.1}(2) follows from Theorem \ref{T1.1}(1).  Thus, we only prove the existence of weak solution below.

 We set $a(t,x,\mu) :=(\si\si^\ast)(t,x,\mu)$ for $t \in [0,T]$, and $b(t, x,\mu) := 0$, $a(t,x,\mu) :=  I$ for $t \in \R\setminus [0,T]$.
Let $0\le \rr\in C_0^\infty(\R\times\R^d)$ with support contained in $\{(r,x): |(r,x)|\le 1\}$ such that $\int_{\R\times\R^d} \rr(r,x)\d r\d x=1.$
For any $n\ge 1$, let $\rr_n(r,x)= n^{d+1} \rr(nr, nx)$ and define
\begin{equation}\begin{split}\label{approx}
&a^n(t,x,\mu)=\int_{\R\times\R^d} \sigma\sigma^\ast(s,x',\mu) \rr_n (t-s, x-x')\d s \d x',\\
&b^n(t,x,\mu)=\int_{\R\times\R^d} b(s,x',\mu) \rr_n (t-s, x-x')\d s \d x',\ \ (t,x,\mu)\in \R\times\R^d\times\scr P.
\end{split}\end{equation}
 Let $\hat{\sigma}^n=\ss{a^n}$ and $\hat{\sigma}=\ss{a}$. Consider the following SDE:
\beq\label{E1''} \d X(t)= b(t,X(t), \L_{X_t})\d t+B(t,X_t,\L_{X_t})\d t +\hat{\si}(t,X(t), \L_{X_t})\d W(t).\end{equation}
Noting that $\sigma\sigma^\ast=\hat{\sigma}\hat{\sigma}^\ast$, in order to prove that the SDE \eqref{E10} has a weak solution, we only need to prove that SDE \eqref{E1''} has a weak solution.

Since by \cite {HW18}, there exist  subsequence $\{n_k\}$ and $G\in L^q_p(T)$ such that $|b^{n_k}|^2\leq K+G$. Below, we use the subsequence $b^{n_k}$ replacing $b^n$. For simplicity, we still denote by $b^n$.

For any $n\ge 1$ there exists a constant  $c_n>0$ such that
\begin{align*}   |b^n(t,x,\mu)-b^n(s,x',\mu)|+ \|\hat{\sigma}^n(t,x,\mu)-\hat{\sigma}^n(s,x',\mu)\|
 \le c_n \big(|t-s|+|x-x'| \big)
\end{align*}
holds for all $s,t\in \R, x,x'\in \R^d$ and $\mu\in \scr P_1$.
This and \eqref{Lipx} imply that the SDE with $X^0(t)=X_0(t\wedge 0)$:
 \beq\label{X^n} \begin{split}\d X^n(t)&=  B(t,X^n_t,\L_{X_t^{n-1}}) \d t+b^n(t,X^n(t),\L_{X^{n-1}_t})\d t+ \si^n(t,X^n(t), \L_{X^{n-1}_t}) \d W(t)\end{split}\end{equation}
 with $X_0^n= X_0$ has a unique strong solution $(X^n_t)_{t\in [0,T]}$.   Moreover, Lemma \ref{KK} implies that for any  $(p,q)\in \scr K$,
\beq\label{KRE} \E  \int_s^t f(r,X^n(r))\d r \le C(t-s)^\dd \|f\|_{L_p^q(T)},\ \ 0\le f\in L_p^q(T), n\ge 1\end{equation} holds for some constants $C>0$ and $ \dd\in (0,1).$

 We first show that Lemma \ref{PC} applies to  $(X^n,W)$ replacing $\psi_n$, for which it suffices to verify conditions \eqref{Ub} and \eqref{ETC}  with $\psi_n:=X^n$. By condition (2) in $(H^\theta)$ and \eqref{APP3} implied by \eqref{APP'}, there exist  constants $c_1,c_2>0$ such that
\begin{equation}\label{AP3}\beg{split}
 \E |X  ^n(t)|^\theta &\leq c_1\bigg\{\E |X(0)|^\theta+\E\bigg(\int_0^T|b^n(t,X^n(t),\L_{X^{n-1}_t})|\,\d t\bigg)^\theta\\
 &\ \ \ \ \ \ +\E\bigg(\int_0^T|B(t,X^n_t,\L_{X_t^{n-1}}) |\,\d t\bigg)^\theta\\
 &\ \ \ \ \ \ \ \ \  +  \E\left(\int_0^T\|\sigma^n(t,X^n(t),\L_{X^{n-1}_t})\|^2\,\d t\right)^\frac{\theta}{2}\bigg\}\\
&\leq c_2\Big(\E |X(0)|^\theta + T^\theta+ \|G\|_{L^{q}_{p}(T)}^\theta + T^{\ff \theta 2}\Big) <\infty,  \ \ n\ge 1, t\in [0,T].
\end{split}\end{equation} Thus, \eqref{Ub} holds for $\psi_n=X^n$ .

Next, by the same reason, there exists a constant $c_3>0$ such that for any $0 \leq s \leq t \leq T$,
\begin{align*}
&\E |X^n(t)-X^n(s)|\\
&\leq \E\int_s^t|b^n(r,X^n(r),\L_{X^{n-1}_r})|\,\d r+\E\int_s^t|B(r,X^n_r,\L_{X_r^{n-1}}) |\,\d r\\
&+  \E\left(\int_s^t\|\sigma^n(r,X^n(r),\L_{X^{n-1}_r})\|^2\,\d r\right)^\frac{1}{2}\\
&\leq c_3\big(t-s + (t-s)^\dd\|G\|_{L^q_p(T)}+ (t-s)^{\frac{1}{2}}\big).
\end{align*}
Hence, \eqref{ETC}   holds for $\psi_n=X^n$ .
According to Lemma \ref{PC}, there exists a subsequence of $(X^n,W)_{n\ge 1}$, denoted again by $(X^n,W)_{n\ge 1}$, stochastic processes $(\tilde{X}^n,\tilde{W}^n)_{n\ge 1}$ and  $(\tilde{X}, \tilde{W})$ on a complete probability space $(\tilde{\OO}, \tilde{\F}, \tilde{\P})$ such that $\L_{(X^n, W)}|\P=\L_{(\tilde{X}^n, \tilde{W}^n)}|\tilde{\P}$ for any $n\geq 1$, and $\tt\P$-a.s. $\lim_{n\to\infty}(\tilde{X}^n, \tt W^n)=(\tilde{X}, \tt W)$. As in \cite{GM}, let $\tilde{\F}^n_{t}$
be  the completion of the $\si$-algebra generated by the  $\{\tilde{X}^n(s), \tilde{W}^n(s): s\leq t\}$. Then as shown in \cite{GM},   $\tilde{X}^n(t)$ is $\tilde{\F}^n_{t}$-adapted and continuous (since $X^n$ is continuous and $\L_{X^n}|\P=\L_{\tt X^n}|\tt\P$),   $\tilde{W}^n$
is a $d$-dimensional Brownian motion on $(\tt \OO, \{\tt \F_t^n\}_{t\in [0,T]},\tt\P)$,    and $(\tilde{X}^n(t),\tilde{W}^n(t))_{t\in [0,T]}$ solves the SDE
\beq\label{titlde-X^n}\begin{split}
\d \tilde{X}^n(t)&= b^n(t,\tilde{X}^n(t),\L_{\tilde{X}^{n-1}_t}|\tilde{\P})\,\d t\\
&+ B(t,\tilde{X}^n_t,\L_{\tilde{X}^{n-1}_t|\tilde{\P}})\,\d t+ \sigma^n(t,\tilde{X}^n(t),\L_{\tilde{X}^{n-1}_t}|\tilde{\P})\,\d \tilde{W}^n(t)
\end{split}\end{equation}
with $\L_{\tilde{X}^n_0}|\tilde{\P}=\L_{X_0}|\P.$
  Simply denote $\L_{\tilde{X}^n_t}|\tilde{\P}=\L_{\tilde{X}^n_t}$ and $\L_{\tilde{X}_t}|\tilde{\P}=\L_{\tilde{X}_t}$.

For any $n\ge 1$, we have
$$\left| \int_{0}^s B(t,\tilde{X}^n_t,\L_{\tilde{X}_t^{n-1}}) \d t-\int_{0}^sB(t,\tilde{X}_t,\L_{\tilde{X}_t}) \d t\right|\le I_1(s)+ I_2(s),$$ where
\begin{align*}
&I_1(s):= \left|\int_0^s B(t,\tilde{X}^n_t,\L_{\tilde{X}_t^{n-1}}) \d t-\int_0^sB(t,\tilde{X}_t,\L_{\tilde{X}_t^{n-1}})\d t\right|,\\
&I_2(s):=\left|\int_0^s B(t,\tilde{X}_t,\L_{\tilde{X}_t^{n-1}}) \d t-\int_0^sB(t,\tilde{X}_t,\L_{\tilde{X}_t}) \d t\right|.
\end{align*}
Below we estimate these $I_i(s)$ respectively.

Firstly, by Chebyshev's inequality, we arrive at
\begin{align*}
\tilde{\P}(\sup_{s\in[0,T]}I_1(s)\geq\varepsilon)&\leq \frac{1}{\varepsilon}\E\int_{0}^T \left|B(t,\tilde{X}^n_t,\L_{\tilde{X}_t^{n-1}})-B(t,\tilde{X}_t,\L_{\tilde{X}_t^{n-1}})\right| \d t\\
\end{align*}
Since $\tilde{\P}$-a.s. $\tilde{X}^n_t$ converges to $\tilde{X}_t$, by \eqref{Lipx} and the boundedness of $B$, we may apply the dominated convergence theorem to derive
\begin{equation}\begin{split}\label{I1J}
&\limsup_{n\to\infty}\tilde{\P}(\sup_{s\in[0,T]}I_1(s)\geq\varepsilon)\leq \frac{1}{\varepsilon}\E\int_{0}^T \lim_{n\to\infty}L_0\|\tilde{X}^n_t-\tilde{X}_t\|_{\C} \d t=0.
\end{split}\end{equation}
Furthermore, since for any $t\in[0,T], \xi\in\C$, $B(t,\xi,\cdot)$ is continuous on $\scr P$, and $\tilde{X}^n_t$ converges to $\tilde{X}_t$ weakly, then it is not difficult to see from Chebyshev's inequality and dominated convergence theorem that 
\begin{align*}
&\limsup_{n\to\infty}\tilde{\P}\Big(\sup_{s\in[0,T]}I_2(s)\geq\varepsilon\Big)\\
&\leq\limsup_{n\to\infty}\frac{1}{\varepsilon}
\int_{0}^T\E\left|B(t,\tilde{X}_t,,\L_{\tilde{X}_t^{n-1}})-B(t,\tilde{X}_t,\L_{\tilde{X}_t})\right|\,\d t=0.
\end{align*}
Thus, we have
\begin{equation*} \lim_{n\to\infty}\tilde{\P}\left(\sup_{s\in[0,T]}\left| \int_{0}^s B(t,\tilde{X}^n_t,\L_{\tilde{X}_t^{n-1}}) \d t-\int_{0}^sB(t,\tilde{X}_t,\L_{\tilde{X}_t}) \d t\right|\geq\varepsilon\right)=0.\end{equation*}
 By \cite[(4.5)-(4.6)]{HW18}, we have
\begin{equation*}
\lim_{n\to\infty}\tilde{\P}\left(\sup_{s\in[0,T]}\int_{0}^s| b^n(t,\tilde{X}^n(t),\L_{\tilde{X}^{n-1}_t}) -b(t,\tilde{X}(t),\L_{\tilde{X}_t})|\,\d t\geq\varepsilon\right)=0,
\end{equation*}
and
\begin{equation*}\lim_{n\to\infty}\tilde{\P}\left(\sup_{s\in[0,T]}\left| \int_{0}^s\si^n(t,\tilde{X}^n(t),\L_{\tilde{X}^{n-1}_t})\d\tilde{W}^n(t)- \int_{0}^s\si(t,\tilde{X}(t),\L_{\tilde{X}_t})\,\d \tilde{W}(t)\right|\geq\varepsilon\right)=0.
\end{equation*}
Then $(\tilde{X}(t),\tilde{W}(t))_{t\in [0,T]}$ is a weak solution to \eqref{E10} by taking limit in \eqref{titlde-X^n}.
\subsection{Uniqueness on Strong Solutions}
In this subsection, we consider uniqueness of strong solutions of \eqref{E10}. To this end, we give the following conditions.
\beg{enumerate} \item[{\bf(A)}] There exist constants $K>1$ and $\theta\ge 1$ such that the following assumptions hold.
\item[{\bf(A1)}]  $\sigma$ is uniformly continuous in $x\in\mathbb{R}^d$ uniformly with respect to $(t,\mu)\in[0,T]\times\scr P_{\theta},$ and for any $(t,x)\in [0,T]\times\R^d, \mu,\nu\in\scr P_\theta$,
    $$K^{-1} I \le (\sigma\sigma^\ast)(t,x,\mu)\le K I, $$
$$  |b(t,x,\mu)-b(t,x,\nu)|+ \|\sigma(t,x,\mu)- \sigma(t,x,\nu)\| \leq K\mathbb{W}_\theta(\mu,\nu).$$
\item[{\bf(A2)}]
 For any $\mu_{\cdot}\in C([0,T]; \scr P_\theta)$, $b^\mu(t,x):= b(t, x, \mu_t)$ and $\si^\mu(t,x):= \si(t,x,\mu_t)$ satisfy
$| b^\mu|^2+\|\nn \si^\mu\|^2 \in L_p^q(T)$ for some $(p,q)\in \scr K$, where $\nn$ is the weak gradient in the space variable $x\in \R^d$.
\item[{\bf(A3)}] $B$ is bounded and for any $t\in[0,T], \xi,\bar{\xi}\in\C,\mu,\nu\in\scr P_\theta$, it holds
 \begin{equation}\label{Lipx'}
 |B(t,\xi,\mu)-B(t,\bar{\xi},\nu)|\leq K(\|\xi-\bar{\xi}\|_{\C}+\mathbb{W}_\theta(\mu,\nu)).
 \end{equation}
\end{enumerate}
We will use the following result for the maximal operator:
\begin{align}\label{max}
\M h(x):=\sup_{r>0}\frac{1}{|B(x,r)|}\int_{B(x,r)}h(y)\d y,\ \  h\in L^1_{loc}(\mathbb{R}^d), x\in \R^d,
 \end{align}
 where $B(x,r):=\{y: |x-y|<r\},$  see   \cite[Appendix A]{CD}.

 \begin{lem} \label{Hardy}  There exists a constant $C>0$ such that for any continuous and weak differentiable function $f$,
 \beq\label{HH1}
|f(x)-f(y)|\leq C|x-y|(\M |\nabla f|(x)+\M |\nabla f|(y)),\ \  {\rm a.e.}\ x,y\in\R^d.\end{equation}
Moreover, for any $p>1$,    there exists a constant $C_{p}>0$ such that
\beq\label{HH2}
\|\M f\|_{L^p}\leq C_{p}\|f\|_{L^p},\ \ f\in L^p(\R^d).
 \end{equation}
\end{lem}

\begin{lem} \label{PU} Assume {\bf(A)}. Let $X$ and $Y$ be two solutions to \eqref{E10} with $X_0=Y_0$, then $\mathbb{P}$-a.s. $X=Y$.
\end{lem}
\begin{proof} Let $\mu_t=\L_{X_t}, \nu_t=\L_{Y_t}, t\in [0,T].$ Then $\mu_0=\nu_0$.
Let $$b^\mu(t,x)= b(t,x, \mu_t),\ \ \ \si^\mu(t,x)= \si(t,x,\mu_t),\ \ (t,x)\in [0,T]\times\R^d,$$ and define  $b^\nu, \si^\nu$ in the same way using $\nu_t$ replacing $\mu_t$. Then
\beq\label{E110'}\beg{split}
&\d X(t)= b^\mu(t,X(t))\,\d t+ B(t,X_t,\mu_t)\,\d t+\sigma^\mu(t,X(t))\,\d W(t),\\
&\d Y(t)= b^\nu(t,Y(t))\,\d t+ B(t,Y_t,\nu_t)\,\d t+\sigma^\nu(t,Y(t))\,\d W(t).\end{split}
\end{equation}
For any $\lambda>0$, consider the following PDE for $u: [0,T]\times\R^d\to \R^d$:
\beq\label{PDE}
\frac{\partial u(t,\cdot)}{\partial t}+\frac{1}{2}\mathrm{Tr} (\sigma^\mu(\sigma^\mu)^\ast\nabla^2u)(t,\cdot)+(\nabla_{b^\mu}u)(t,\cdot)+b^\mu(t,\cdot)=\lambda u(t,\cdot),\ \ u(T,\cdot)=0.
\end{equation} \ By \cite[Theorem 5.1]{Z2},
when $\ll$ is large enough,   \eqref{PDE} has a unique solution $\mathbf{u}^{\lambda,\mu}$ satisfying
\begin{align}\label{u0}
\|\nabla \mathbf{u}^{\lambda,\mu}\|_{\infty}\leq \frac{1}{5},
\end{align}
and \beq\label{u01} \|\nabla^2 \mathbf{u}^{\lambda,\mu}\|_{L^{2q}_{2p}(T)}<\infty.\end{equation}
Let $\theta^{\lambda,\mu}(t,x)=x+\mathbf{u}^{\lambda,\mu}(t,x)$. By \eqref{E10}, \eqref{PDE},  and using the It\^o formula and an approximation technique (see \cite[Lemma 4.3]{Z2} for more details), we derive
\beq\label{E-X}\begin{split}
\d \theta^{\lambda,\mu}(t,X(t))&= \lambda \mathbf{u}^{\lambda,\mu}(t,X(t))\d t+\nabla\theta^{\lambda,\mu}(t,X(t)) B(t,X_t,\mu_t)\d t\\
&\qquad +(\nabla\theta^{\lambda,\mu}\sigma^\mu)(t,X(t))\,\d W(t),
\end{split}\end{equation}
and
\beq\begin{split}\label{E-Y}
\d \theta^{\lambda,\mu}(t,Y(t))&=\lambda \mathbf{u}^{\lambda,\mu}(t,Y(t))\d t+(\nabla\theta^{\lambda,\mu}\sigma^\nu)(t,Y(t))\,\d W(t)\\
&+\nabla\theta^{\lambda,\mu}(t,Y(t)) B(t,Y_t,\nu_t)\d t\\
&+[\nabla\theta^{\lambda,\mu}(b^\nu-b^\mu)](t,Y(t))\d t\\
&+\frac{1}{2}\mathrm{Tr} [(\sigma^\nu(\sigma^\nu)^\ast-\sigma^\mu(\sigma^\mu)^\ast)\nabla^2\mathbf{u}^{\lambda,\mu}](t,Y(t))\d t.
\end{split}\end{equation}
  Let $\xi_t=\theta^{\lambda,\mu}(t,X(t))-\theta^{\lambda,\mu}(t,Y(t))$. By \eqref{E-X}, \eqref{E-Y} and It\^o formula, we obtain
\begin{equation*}\begin{split}
\d|\xi_t|^2
=&2\lambda\left<\xi_t,\mathbf{u}^{\lambda,\mu}(t,X(t))-\mathbf{u}^{\lambda,\mu}(t,Y(t))\right\>\d t\\
&+2\left<\xi_t,\nabla\theta^{\lambda,\mu}(t,X(t)) B(t,X_t,\mu_t)-\nabla\theta^{\lambda,\mu}(t,Y(t)) B(t,Y_t,\nu_t)\right\>\d t\\
&+2\left\<\xi_t,[(\nabla\theta^{\lambda,\mu}\sigma^\mu)(t,X(t))-(\nabla\theta^{\lambda,\mu}\sigma^\nu)(t,Y(t))]\d W(t)\right\>\\
&+\left\|(\nabla\theta^{\lambda,\mu}\sigma^\mu)(t,X(t))-(\nabla\theta^{\lambda,\mu}\sigma^\nu)(t,Y(t))\right\|^2_{HS}\,\d t\\
&-2\left\<\xi_t, \nabla\theta^{\lambda,\mu}(b^\nu-b^\mu)](t,Y(t))\right\>\d t\\
&-\left\<\xi_t,\mathrm{Tr} [(\sigma^\nu(\sigma^\nu)^\ast-\sigma^\mu(\sigma^\mu)^\ast)\nabla^2\mathbf{u}^{\lambda,\mu}](t,Y(t))\right\>\d t.
\end{split}\end{equation*}
So, for any $m\ge 1$,
\begin{equation*}\beg{split}
\d|\xi_t|^{2m}
 =\, &2m\lambda|\xi_t|^{2(m-1)}\left<\xi_t,\mathbf{u}^{\lambda,\mu}(t,X(t))-\mathbf{u}^{\lambda,\mu}(t,Y(t))\right\>\d t\\
 &+2m|\xi_t|^{2(m-1)}\left<\xi_t,\nabla\theta^{\lambda,\mu}(t,X(t)) B(t,X_t,\mu_t)-\nabla\theta^{\lambda,\mu}(t,Y(t)) B(t,Y_t,\nu_t)\right\>\d t\\
&+2m|\xi_t|^{2(m-1)}\left\<\xi_t,[(\nabla\theta^{\lambda,\mu}\sigma^\mu)(t,X(t))-(\nabla\theta^{\lambda,\mu}\sigma^\nu)(t,Y(t))]\d W(t)\right\>\\
&+m|\xi_t|^{2(m-1)}\left\|(\nabla\theta^{\lambda,\mu}\sigma^\mu)(t,X(t))-(\nabla\theta^{\lambda,\mu}\sigma^\nu)(t,Y(t))\right\|^2_{HS}\,\d t\\
&+2m(m-1) |\xi_t|^{2(m-2)}\left|[(\nabla\theta^{\lambda,\mu}\sigma^\mu)(t,X(t))-(\nabla\theta^{\lambda,\mu}\sigma^\nu)(t,Y(t))]^\ast\xi_t \right|^2\d t\\
&-2m|\xi_t|^{2(m-1)}\left\<\xi_t, \nabla\theta^{\lambda,\mu}(b^\nu-b^\mu)](t,Y(t))\right\>\d t\\
&-m|\xi_t|^{2(m-1)}\left\<\xi_t,\mathrm{Tr} [(\sigma^\nu(\sigma^\nu)^\ast-\sigma^\mu(\sigma^\mu)^\ast)\nabla^2\mathbf{u}^{\lambda,\mu}](t,Y(t))\right\>\d t.\end{split}\end{equation*}
Firstly,
\beq\label{XPP0}\begin{split}
&|\xi_t|^{2(m-1)}\left<\xi_t,\nabla\theta^{\lambda,\mu}(t,X(t)) B(t,X_t,\mu_t)-\nabla\theta^{\lambda,\mu}(t,Y(t)) B(t,Y_t,\nu_t)\right\>\d t\\
&\le c_0 |\xi_t|^{2m}\scr M\big(\|\nn^2\theta^{\ll,\mu}\|(t, X(t))+\|\nn^2\theta^{\ll,\mu}\|(t, Y(t))\big)\\
&+c_0 \sup_{s\in[0,t]}|\xi_s|^{2m}+c_0\W_\theta(\mu_t,\nu_t)^{2m}.
\end{split}\end{equation}
According to \cite{HW18} and  \eqref{XPP0}, we arrive at
\beq\label{NNP}\d |\xi_t|^{2m} \le c_2 \sup_{s\in[0,t]}|\xi_s|^{2m} \d t +c_2 |\xi_t|^{2m} \d A_t+c_2 \W_\theta(\mu_t,\nu_t)^{2m}\d t + \d M_t\end{equation}
for some constant $c_2>0$, a local martingale $M_t$,  and
\begin{align*}
A_t:=c\int_0^t\Big\{&1+ |\nn^2{\mathbf u}^{\ll,\mu}(s,Y(s))|^2 +\big(\scr M\big(\|\nn^2\theta^{\ll,\mu}\|+\|\nn\si^\mu\|\big)(s,X(s))  \\
&+  \scr M\big(\|\nn^2\theta^{\ll,\mu}\|+\|\nn\si^\mu\|\big)(s,Y(s))\big)^2\Big\}\d s.
\end{align*}
By It\^{o}'s formula, we have
\begin{equation*}\d \e^{-A_t}|\xi_t|^{2m} \le c_2 \e^{-A_t}\sup_{s\in[0,t]}|\xi_s|^{2m} \d t +c_2 \e^{-A_t} \W_\theta(\mu_t,\nu_t)^{2m}\d t + \e^{-A_t}\d M_t\end{equation*}
When $2m>\theta$, we can take $p\in(0,1)$ such that $2mp>\theta$. By the stochastic Gr\"{o}nwall lemma due to \cite[Lemma A.5]{B}, we arrive at
 \beq\label{NN2}\begin{split}\W_\theta(\mu_t,\nu_t)^{2m}&\leq c_3\left(\E \sup_{s\in[0,t]}|\xi_s|^\theta\right)^{\frac{2m}{\theta}}\leq c_3\left(\E\e^{\frac{\theta}{2m}A_t} \sup_{s\in[0,t]}\e^{-\frac{\theta}{2m}A_s}|\xi_s|^\theta\right)^{\frac{2m}{\theta}}\\
 &\le c_3 \left(\E\e^{\ff{2mp}{2mp-\theta}\frac{\theta}{2m}A_t}\right)^{\ff{2mp-\theta}{p\theta}} \left(\E \left(\sup_{s\in[0,t]}\e^{-\frac{\theta}{2m}A_s}|\xi_s|^\theta\right)^{\frac{2mp}{\theta}}\right)^{\ff{1}{p}}\\
 &=c_3 \left(\E\e^{\ff{\theta p}{2mp-\theta}A_t}\right)^{\ff{2mp-\theta}{p\theta}}  \left(\E \left(\sup_{s\in[0,t]}\e^{-A_s}|\xi_s|^{2m}\right)^{p}\right)^{\ff{1}{p}}\\
  &\leq c_4 \left(\E\e^{\ff{\theta p}{2mp-\theta}A_t}\right)^{\ff{2mp-\theta}{p\theta}}  \int_0^t\W_\theta(\mu_s,\nu_s)^{2m}\d s,\ \ t\in [0,T]\end{split}\end{equation}
for some constants $c_3, c_4>0$.
Since by Lemma \ref{KK}, \eqref{HH2}, \eqref{u01} and the Khasminskii type estimate, see for instance \cite[Lemma 3.5]{XZ}, we have
$$\E\e^{\ff{\theta p}{2mp-\theta}A_T}<\infty,$$
so that by the Gr\"{o}nwall lemma we prove $\W_\theta(\mu_t,\nu_t)=0$ for all $t\in [0,T].$ Then by \eqref{NN2}, we conclude $\xi_t=0$, which implies $X_t=Y_t$ for all $t\in [0,T].$
\end{proof}
\subsection{Proof of Theorem \ref{T1.1}(3)}
Under the assumption of Theorem \ref{T1.1}(3), applying Lemma \ref{PU}, we get the uniqueness of strong solution of \eqref{E10}. This and Theorem \ref{T1.1}(2) imply \eqref{E10} has a unique strong solution.
 \section{Proofs of Theorems \ref{T3.1}-\ref{T5.1}}

 \subsection{Proof of Theorem \ref{T3.1}}
By \cite[Theorem 1.1]{HZ}  with  $\mathbb{H}=\mathbb{R}^d$,  we know that \eqref{class} has a unique strong solution $X_t$ up to life time. Combining this with Themrem \ref{T1.1}, Lemma \ref{SS},  we see that the SDE \eqref{E10} has strong   existence and uniqueness under {\bf(H)}.  For any $\mu\in \scr P_2$ we let $\mu_t=P_t^*\mu$ be the distribution of $X_t$
which solves \eqref{E11} with $\L_{X_0}=\mu.$
\beg{proof}[Proof of Theorem \ref{T3.1}] 
For $\mu_t:= P_t^*\mu_0$ and $\nu_t:=P_t^*\nu_0$, we may
rewrite \eqref{E11} as
\begin{equation}\label{barX}
\d X(t)= \bar{b}(t,X(t))\d t +\bar{B}(t,X_t)\d t+\si(t,X(t))\d \bar{W}(t),\ \ \scr L_{X_0}=\mu_0,
\end{equation}
where
\begin{equation*}\begin{split}
&\bar{b}(t,x):=b(t,x,\nu_t),\ \ \bar{B}(t,\xi):=B(t,\xi,\nu_t),\ \ \d \bar{W}(t):=\d W(t)+ \bar{\gamma}(t)\d t,\\
& \bar{\gamma}(t):=[\si^\ast(\si\si^\ast)^{-1}](t,X(t))[b(t,X(t),\mu_t)- b(t,X(t),\nu_t)+B(t,X_t,\mu_t)- B(t,X_t,\nu_t)].
\end{split}\end{equation*}
By assumption {\bf (H)}, using $$\d A_t\leq C\d t$$
in the proof of Lemma \ref{PU}, we have
\begin{equation}\begin{split}\label{EbarG}
|\bar{\gamma}(t)|\leq C\W_2(\mu_t,\nu_t)\le K(t) \W_2(\mu_0,\nu_0), \ \ t\in[0,T]
\end{split}\end{equation} for some increasing function $K:\R_+\to\R_+.$
Let
\begin{equation}\label{EB2}
\bar{R}_t=\exp\left\{-\int_0^t\langle\bar{\gamma}(s),\d W(s)\rangle-\frac{1}{2}\int_0^t|\bar{\gamma}(s)|^2\d s\right\},\ \ t\in[0,T].
 \end{equation}
  By Girsanov's theorem, $\{\bar{W}(t)\}_{t\in[0,T]}$ is a $d$-dimensional Brownian motion under the probability measure $\bar{\P}_T:= \bar R_T\P$.

Next, according to the proof of \cite[Lemma 3.2]{HW}, we can construct an adapted process $\tt\gg(t)$ on $\R^d$ such that
\beg{enumerate} \item[(a)] Under the probability measure $\bar \P_T$,
$$\tilde{R}_t:=\exp\left\{-\int_0^t\langle\tilde{\gamma}(s),\d \bar{W}(s)\rangle-\frac{1}{2}\int_0^t|\tilde{\gamma}(s)|^2\d s\right\}, \ \ t\in[0,T]$$
is a martingale, such that $\tt\P_T:= \tt R_T\bar\P_T=\tt R_T\bar R_T\P$ is a probability measure under which
$$\tt W(t):=  \bar{W}(t)+ \int_0^t \tilde{\gamma}(s)\d s =  W(t)+ \int_0^t \big(\bar{\gamma}(s)+\tilde{\gamma}(s)\big)\d s,\ \ t\in [0,T]$$ is a $d$-dimensional Brownian motion.
\item[(b)] Letting $Y(t)$ solve the following stochastic functional differential equation under the probability measure $\tt\P_T$ with the given initial value $Y_0$:
\begin{equation*}\begin{split}
\d Y(t) &= \bar{b}(t,Y(t))\d t+ \bar{B}(t,Y_t)\d t+\si(t,Y(t)) \d \tilde{W}(t),
\end{split}\end{equation*}  we have  $\L_{Y_0|\tt\P}=\L_{Y_0}=\nu_0$ and  $X_T=Y_T\ \tt\P_T$-a.s..
\item[(c)] There exists  $C\in C(\R_+;\R_+)$ such that
$$ \E_{\tt\P_T} \int_0^T |\tt\gg(s)|^2\d s \le C(T) \E\Big(\ff{|X(0)-Y(0)|^2}{T-r}+\|X_0-Y_0\|_{\C}^2\Big).$$
    \end{enumerate}

By the definition of $\bar b$ we see that $(Y_t,\tt W(t))$ is a weak solution to the equation \eqref{barX} with initial distribution $\nu_0$, so that by the weak uniqueness,
$\scr L_{Y_t}|_{\tt \P_T}=\nu_t, t\in [0,T].$ Combining this with (b) we obtain
$$(P_Tf)(\nu_0)= \E_{\tt \P_T} [f(Y_T)] = \E_{\tt P_T}[f(X_T)] = \E [\bar R_T\tt R_T f(X_T)],\ \ f\in \B_b^+(\C).$$
Letting $R_T= \bar R_T\tt R_T$,   by Young's inequality and H\"older's inequality respectively, we obtain
\beq\label{LHI}  (P_T\log f)(\nu_0) \le \E [  R_T \log   R_T ]+ \log\E[f(X_T)]
 = \E [  R_T \log   R_T ]+ \log(P_T f)(\mu_0),
 \end{equation}
and
\beq\label{HI}\begin{split}
&(P_T f(\nu_0))^{p}\le (\E R_T^{\ff {p}{p-1}})^{p-1} (\E f^{p}(X_T))= (\E R_T^{\ff {p}{p-1}})^{p-1} P_T  f^{p}(\mu_0),\ \ p>1.
\end{split}\end{equation}
 We are now ready to prove the Harnack inequality.

 By \eqref{EbarG} ,  (c)   and since $\W_2(\mu_0,\nu_0)^2\le \E\|X_0-Y_0\|_{\C}^2$,
\begin{equation*}\beg{split}
\E[R_T\log R_T]&\le \frac{1}{2}\E_{\tilde{\P}_T}\int_0^T|\bar{\gamma}(s)+\tilde{\gamma}(s)|^2\d s\\
&\le \E_{\tilde{\P}_T}\int_0^T|\tilde{\gamma}(s)|^2\d s+ \int_0^T|\bar{\gamma}(s)|^2\d s\\
&\le \E_{\tilde{\P}_T}\int_0^T|\tilde{\gamma}(s)|^2\d s+\int_0^TC(t)\W_2(\mu_t,\nu_t)^2\d t\\
&\le H_1(T) \mathbb{E}\bigg(\ff{|X(0)-Y(0)|^2}{T-r}+ \|X_0-Y_0\|_{\C}^2\bigg), \ \ T>r\end{split}\end{equation*}
holds for some $H_1\in C(\R_+;\R_+).$
Combining this with \eqref{LHI} we obtain the Log-Harnack inequality.

 Finally, according to the proof of \cite[Theorem 4.1]{HRW}, there exists $p_0>1$ such that for any $p>p_0$,
$$(\E_{\bar\P_T} \tt R_T^{\ff {p}{p-1}})^{\frac{p-1}{p}} \le   \E  \Big(\e^{C(p,T)\big(1+\ff{|X(0-Y(0)|^2}{T-r} +\|X_0-Y_0\|_{\C}^2\big)}\Big),\ \ T>r.  $$
By applying this estimate for $p_1:=\ff 1 2(p+(p_0)$ and combining with $R_T=\tt R_T \bar R_T$, \eqref{EbarG}, \eqref{EB2} and
$\W_2(\mu_0,\nu_0)^2\le \E\|X_0-Y_0\|_{\C}^2,$ we arrive at
\beg{align*} \Big(\E  R_T^{\ff {p}{p-1}}\Big)^{\frac{p-1}{p}} &=\Big(\E_{\bar\P_T}   \tt R_T^{\ff {p}{p-1}} \bar R_T^{\ff 1 {p-1}}\Big)^{\frac{p-1}{p}}
 \le  \Big(\E_{\bar \P_T} \tt R_T^{\ff {p_1}{p_1-1}}\Big)^{\frac{p_1-1}{p_1}} \Big(\E_{\bar\P_T} \bar R_T^{\ff {p_1}{p-p_1}}\Big)^{\frac{p-p_1}{pp_1}}\\
&\le  \E  \Big(\e^{C(p_1,T)\big(1+\ff{|X(0-Y(0)|^2}{T-r} +\|X_0-Y_0\|_{\C}^2\big)}\Big) \Big(\E\bar R_T^{\ff {p}{p-p_1}}\Big)^{\frac{p-p_1}{pp_1}}\\
&\le \E \Big(\e^{H_2(p,T)\big(1+\ff{|X(0-Y(0)|^2}{T-r} +\|X_0-Y_0\|_{\C}^2\big)}\Big),\ \ T>r\end{align*} for some $H_2\in C(\R_+;\R_+).$ Therefore, \eqref{H2'} follows from \eqref{HI}.
\end{proof}
\subsection{Proof of Theorem \ref{T5.1}}
\begin{proof}[Proof of Theorem \ref{T5.1}] Again by the semigroup property and Jensen's inequality, we only need to consider $T-r\in (0,1]$. Define
 $$\gamma(s):= \beg{cases} \ff{s^+}{T-r} \eta(-r),\ \  &\text{if}\  s\in[-r,T-r],\\
\eta(s-T), &\text{if}\
s\in(T-r,T].\end{cases} $$
Next, we construct couplings.
For fixed $\mu_0\in\scr P_2$, let $X(t)$ solve \eqref{E5} with $\L _{X_0}= \mu_0$; and let $\bar{X}(t)$ solve the equation
\beq\label{EC1}
\d \bar{X}(t)= \{b(t,X(t),\mu_t)+B(t,X_t,\mu_t)\}\d t +\si(t,\mu_t) \d W(t)+ \gamma'(t)\d t\end{equation} with
$\bar{X}_0= X_0$. Then the solution to (\ref{EC1}) is non-explosive as well. Moreover,
\begin{equation*} \bar{X}(s)=X(s)+\gamma(s), \ \
s\in[-r,T].\end{equation*}
In particular, $$\bar{X}_T=X_T+\eta.$$
By the definitions of $\gamma$,  there exists a constant
$C>0$ such that for any $s\in[0,T]$,
\beg{equation}\label{NN0'}\beg{split}
&|\gamma'(s)|\le C1_{[0,T-r]}(s)\ff {|\eta(-r)|}{T-r}+C1_{[T-r,T]}(s)|\eta'(s-T)|,\\
&|\gamma(s)|\le C|\eta(-r)|+C\|\eta\|_{\C}\leq C\|\eta\|_{\C},\ \ \|\gamma_s\|_{\C} \le C\|\eta\|_{\C}.\end{split}\end{equation}
Let
$$
\bar{\Phi}(s)=b(s,X(s),\mu_s)-b(s,\bar{X}(s),
\mu_s)+B(s,X_s,\mu_s)-B(s,\bar{X}_s,\mu_s)+\gamma'(s).
$$
From {\bf (H)} and \eqref{NN0'}, it holds
\begin{equation}\begin{split}\label{Phi'}
\int_0^T|\bar{\Phi}(s)|^2\d s&\leq C\int_0^T\left(\phi(|\gamma(s)|)+\|\gamma_s\|_{\infty}+|\gamma'(s)|\right)^2\d s\\
&\leq C\ff {|\eta(-r)|^2}{T-r}+C\int_{-r}^0|\eta'(s)|^2\d s+CT\phi^2\left(C\|\eta\|_{\C}\right)+CT\|\eta\|_{\C}^2.
\end{split}\end{equation}
Set
\begin{align*}
\bar{R}(s)=\exp\bigg[-\int_0^s\< \si^{-1}(u,\mu_u)\bar{\Phi}(u), \d
W(u)\>-\frac{1}{2}\int_0^s |\si^{-1}(u,\mu_u)\bar{\Phi}(u)|^2\d u\bigg],
\end{align*}
and
$$
\bar{W}(s)=W(s)+\int_0^s\si^{-1}(u,\mu_u)\bar{\Phi}(u)\d u.
$$
Girsanov's theorem implies that $\bar{W}$ is a Brownnian motion on $[0,T]$ under $\bar{\Q}_T=\bar{R}(T)\P$. Then (\ref{EC1}) reduces to
\begin{equation*}
\d \bar{X}(t)= \{b(t,\bar{X}(t),\mu_t)+B(t,\bar{X}_t,\mu_t)\}\d t +\si(t,\mu_t) \d \bar{W}(t).
\end{equation*}
Thus, the distribution of $\bar{X}_T$ under $\bar{\Q}_T$ coincides with that of $X_T$  under $\P$.

On the other hand, by Young's inequality,
\begin{align*}
P_T \log f(\mu_0)&=\E^{\bar{\Q}_T}\log f(\bar{X}_T)\\
&=\E ^{\bar{\Q}_T}\log f(X_T+\eta)\\
&\leq \log P_T f(\cdot+\eta)(\mu_0)+\E \bar{R}(T)\log \bar{R}(T),
\end{align*}
and by H\"{o}lder inequality,
\begin{align*}
P_T f(\mu_0)&=\E^{\bar{\Q}_T}f(\bar{X}_T)\\
&=\E ^{\bar{\Q}_T}f(X_T+\eta)\leq (P_T f^p(\cdot+\eta))^{\frac{1}{p}}(\mu_0)\{\E \bar{R}(T)^{\frac{p}{p-1}}\}^{\frac{p-1}{p}}.
\end{align*}
Since $\bar W$ is a Brownian motion under $\bar{\Q}_T$, by the definition of $\bar{R}(T)$, it is easy to see that
\begin{align*}
\E \bar{R}(T)\log \bar{R}(T)=\E ^{\bar{\Q}_T}\log \bar{R}(T)=\frac{1}{2}\E \int_0^T |\si^{-1}(u,\mu_u)\bar{\Phi}(u)|^2\d u\leq \beta(T,\eta,r),
\end{align*}
and
\begin{align*}
&\E \bar{R}(T)^{\frac{p}{p-1}}\\
&\leq \E\Bigg\{\exp\bigg[\frac{-p}{p-1}\int_0^T\< (\sigma^{-1}(u,\mu_u)\bar{\Phi}(u), \d
W(u)\>\\
&-\frac{1}{2}\frac{p^2}{(p-1)^2}\int_0^T |\sigma^{-1}(u,\mu_u)\bar{\Phi}(u)|^2\d u\bigg]\\
&\ \ \ \ \ \ \ \ \ \times\exp\bigg[\frac{1}{2}\frac{p^2}{(p-1)^2}\int_0^T |\sigma^{-1}(u,\mu_u)\bar{\Phi}(u)|^2\d u\\
&-\frac{1}{2}\frac{p}{p-1}\int_0^T |\sigma^{-1}(u,\mu_u)\bar{\Phi}(u)|^2\d u\bigg]\Bigg\}\\
&\leq\mathrm{ess}\sup_{\Omega}\exp\left\{\frac{p}{2(p-1)^2}\int_0^T |\sigma^{-1}(u,\mu_u)\bar{\Phi}(u)|^2\d u\right\}.
\end{align*}
Combining \eqref{Phi'}, the shift Harnack inequality holds.
 \end{proof}


\begin{thebibliography}{999}
\bibitem{B} S. Bachmann, \emph{Well-posedness and stability for a class of stochastic delay differential equations with singular drift,} 18(2018), Stochastics and Dynamics, 18(2018), 27pp.
\bibitem{CA} K. Carrapatoso, \emph{Exponential convergence to equilibrium for the homogeneous Landau equation with hard potentials,} Bull. Sci. Math. 139(2015), 777-805.


\bibitem{CD} G. Crippa, C. De Lellis, \emph{Estimates and regularity results for the DiPerna- Lions flow,} J. Reine Angew. Math. 616(2008), 15-46.

\bibitem{CR} P. E. Chaudru de Raynal, \emph{Strong well-posedness of McKean-Vlasov stochastic differential equation with H\"older drift, } arXiv: 1512.08096.


\bibitem{DV1} L. Desvillettes, C. Villani, \emph{On the spatially homogeneous Landau equation for hard potentials, Part I: existence, uniqueness and smothness,}  Comm. Part. Diff. Equat. 25(2000), 179--259.

\bibitem{DV2}  L. Desvillettes, C. Villani, \emph{On the spatially homogeneous Landau equation for hard potentials, Part II: H-Theorem and Applications,} Comm. Part. Diff. Equat. 25(2000), 261--298.

\bibitem{GM} L. Gyongy, T. Martinez,  \emph{On stochastic differential equations with locally unbounded drift,}  Czechoslovak Math.J. 51(2001), 763-783.

\bibitem{HRW} X. Huang, M. R\"{o}ckner, F.-Y. Wang, \emph{Nonlinear Fokker--Planck equations for probability measures on path space and path-distribution dependent SDEs,}  arXiv:1709.00556.

\bibitem{HW} X. Huang, F.-Y. Wang, \emph{Functional SPDE with Multiplicative Noise and Dini Drift,}  Ann. Fac. Sci. Toulouse Math. 6(2017), 519-537.

\bibitem{HW18} X. Huang, F.-Y. Wang, \emph{Distribution Dependent SDEs with Singular Coefficients,}  https://doi.org/10.1016/j.spa.2018.12.012.

\bibitem{HZ} X. Huang, S.-Q. Zhang, \emph{Mild Solutions and Harnack Inequality for Functional SPDEs with Dini Drift,} J. Theor. Probab. 32(2019), 303-329.

\bibitem{KR} N. V. Krylov, \emph{Controlled diffusion processes,}  Springer, New York, 1980.
\bibitem{MV} Yu. S. Mishura, A. Yu. Veretennikov, \emph{Existence and uniqueness theorems for solutions of McKean-Vlasov stochastic equations,} arXiv:1603.02212.




\bibitem{RW10} M. R\"ockner, F.-Y. Wang, \emph{Harnack and functional inequalities for generalized Mehler semigroups,}  J. Funct. Anal.  203(2007), 237-261.

\bibitem{RZ} M. R\"ockner, X. Zhang, \emph{Well-posedness of distribution dependent SDEs with singular drifts,} arXiv:1809.02216.

\bibitem{SZ} A.-S. Sznitman,   \emph{Topics in propagation of chaos,} In $``$\'Ecole d'\'Et\'e de Probabilit\'es de Sain-Flour XIX-1989", Lecture Notes in Mathematics  1464, p. 165-251, Springer, Berlin, 1991.



 \bibitem{W14a} 	F.-Y. Wang, \emph{Integration by parts formula and shift Harnack inequality for stochastic  equations,}   Ann. Probab. 42(2014), 994-1019.

\bibitem{Wbook} F.-Y. Wang, \emph{Harnack Inequality and Applications for Stochastic Partial Differential Equations,} Springer, New York, 2013.

\bibitem{W16} F.-Y. Wang,  \emph{Gradient estimate and applications for SDEs in Hilbert space with multiplicative noise and Dini continuous drift,}  JDE, 260 (2016), 2792-2829.
    \bibitem{FYW1} F.-Y. Wang, \emph{Distribution-dependent SDEs for Landau type equations,} Stoch. Proc. Appl. 128(2018), 595-621.


\bibitem{XZ} L. Xie, X. Zhang, \emph{Ergodicity of stochastic differential equations with jumps and singular coefficients,}   arXiv:1705.07402.











\bibitem{Z2} X. Zhang, \emph{Stochastic homeomorphism flows of SDEs with singular drifts and Sobolev diffusion coefficients,} Electron. J. Probab. 16(2011), 1096-1116.



\bibitem{AZ} A. K. Zvonkin,  \emph{A transformation of the phase space of a diffusion process that removes the drift,}  Math. Sb. 93(1974), 129-149.




\end{thebibliography}
\end{document}